\documentclass{article}
\usepackage{amssymb}
\usepackage{amsthm}
\usepackage[curve,matrix,arrow]{xy}
\theoremstyle{plain}\newtheorem{Theorem}{Theorem}[section]
\theoremstyle{plain}
\theoremstyle{plain}\newtheorem{Corollary}[Theorem]{Corollary}
\theoremstyle{plain}\newtheorem{Lemma}[Theorem]{Lemma}
\theoremstyle{plain}\newtheorem{Proposition}[Theorem]{Proposition}
\theoremstyle{definition}
\theoremstyle{definition}
\theoremstyle{definition}
\theoremstyle{definition}\newtheorem{Remark}[Theorem]{Remark}
\theoremstyle{definition}
\theoremstyle{plain}\newtheorem{Equation}[Theorem]{}

           \def\tenk{\otimes_k}     
             \def\ten{\otimes}

\def\coker{\mathrm{coker}}

\def\Endbar{\underline{\mathrm{End}}}
\def\Ext{\mathrm{Ext}}           

\def\hatExt{\widehat{\mathrm{Ext}}} 
\def\barExt{\overline{\mathrm{Ext}}}
\def\hatHH{\widehat{HH}} \def\barHH{\overline{HH}}
\def\Hom{\mathrm{Hom}}           
\def\Hombar{\underline{\mathrm{Hom}}}
\def\ker{\mathrm{ker}}           
\def\Id{\mathrm{Id}}             \def\tenA{\otimes_A}
\def\Im{\mathrm{Im}}             \def\tenB{\otimes_B}
\def\Ind{\mathrm{Ind}}

\def\Mod{\mathrm{Mod}}           
\def\mod{\mathrm{mod}}      \def\modbar{\underline{\mathrm{mod}}}
\def\perf{\mathrm{perf}}    \def\perfbar{\underline{\mathrm{perf}}}

\def\res{\mathrm{res}}         
           
\def\Tr{\mathrm{Tr}}             
\def\tr{\mathrm{tr}}

\title{Tate duality and transfer in Hochschild cohomology} 
\author{Markus Linckelmann} 
\date{}

\begin{document}

\maketitle

\begin{abstract}
We show that dualising transfer maps in Hochschild cohomology
of symmetric algebras commutes with Tate duality. This extends 
a well-known result in group cohomology.
\end{abstract}

\section{Introduction}

Let $k$ be a field. For $V$ a $k$-vector space, 
denote by $V^\vee$ its $k$-dual $\Hom_k(V,k)$. A  finite-dimensional 
$k$-algebra $A$ is called {\it symmetric} if $A\cong$ $A^\vee$  as 
$A$-$A$-bimodules. The image $s\in$ $A^\vee$ of $1_A$ under such an 
isomorphism is called a {\it symmetrising form for} $A$. 
It is well-known that the Tate analogue  $\hatHH^*(A)$ of the 
Hochschild cohomology of a symmetric $k$-algebra $A$ satisfies a 
duality $(\hatHH^{-n}(A))^\vee\cong$  $\hatHH^{n-1}(A)$, for 
every integer $n$. If $A$, $B$ are two symmetric algebras and $M$ 
is an $A$-$B$-bimodule which is finitely generated as a left 
$A$-module and as a right $B$-module, then $M$ induces a transfer 
map $\tr_M : \hatHH^*(B)\to$ $\hatHH^*(A)$, and the dual 
$M^\vee$ induces a transfer map $\tr_{M^\vee} : \hatHH^*(A)\to$ 
$\hatHH^*(B)$. These transfer maps depend on the choices of
symmetrising forms for $A$ and $B$. In positive degree they 
coincide with the transfer maps constructed in \cite{Lintransfer}.
Similarly, for any two finitely generated $B$-modules $V$, $W$,
there is a transfer map $\tr_{M^\vee}=$ $\tr_{M^\vee}(V,W) :
\hatExt_A^n(M\tenB V,M\tenB W)\to$ $\hatExt_B^n(V,W)$.
The underlying constructions are special cases of a general principle 
associating transfer maps with pairs of adjoint functors between
triangulated categories; see 
\cite[\S 7]{Ligrblock} or Section \ref{Tatetransfersection} below 
for a brief review.

\begin{Theorem} \label{transferTatedual}
Let $A$, $B$ be symmetric $k$-algebras, and let $M$ be an 
$A$-$B$-bimodule which is finitely generated projective as
a left $A$-module and as a right $B$-module. For any integer 
$n$ we have a commutative diagram of $k$-vector spaces
$$\xymatrix{ 
\hatHH^{n-1}(A) \ar[rr]^{\tr_{M^\vee}} \ar[d] & & 
\hatHH^{n-1}(B) \ar[d] \\
(\hatHH^{-n}(A))^\vee \ar[rr]_{(\tr_M)^\vee} & &
(\hatHH^{-n}(B))^\vee  }$$
where the vertical maps are the Tate duality isomorphisms.
\end{Theorem} 

Theorem \ref{transferTatedual} holds with $M$ replaced by
a bounded complex of $A$-$B$-bimodules $X$ whose components $X_i$
are finitely generated projective as left and right modules.
This follow, for instance, from the formula $\tr_X=$ $\sum_{i} (-1)^i \tr_{X_i}$ 
in \cite[2.11. (ii)]{Lintransfer}, with $i$
running over the integers for which $X_i$ is nonzero. 
Alternatively, if $U$ is a finitely generated projective $A$-$B$-bimodule, then
$\tr_U$ is zero on $\hatHH^*(B)$. By a standard argument due to Rickard
(appearing at the end of the proof of \cite[2.1]{RickStable}),
$X$ is quasi-isomorphic to a complex with at most one nonprojective 
component. Thus there is a bimodule $M$ such that $\tr_X=$ $\tr_M$ on 
$\hatHH^*(B)$. In particular, there is no loss of generality in
stating Theorem \ref{transferTatedual} for bimodules rather than 
complexes.

\begin{Theorem} \label{transferVWTatedual}
Let $A$, $B$ be symmetric $k$-algebras, and let $M$ be an 
$A$-$B$-bimodule which is finitely generated projective as
a left $A$-module and as a right $B$-module. Let $V$, $W$ be
finitely generated $B$-modules. For any integer 
$n$ we have a commutative diagram of $k$-vector spaces
$$\xymatrix{
\hatExt^{n-1}_B(V,W) \ar[rr]^(.4){M\tenB-}\ar[d] & & 
\hatExt^{n-1}_A(M\tenB V,M\tenB W) \ar[rr]^(.6){\tr_{M^\vee}} \ar[d] & & 
\hatExt^{n-1}_B(V,W) \ar[d] \\
\hatExt_B^{-n}(W,V)^\vee \ar[rr]_(.4){(\tr_{M^\vee})^\vee} & & 
\hatExt_A^{-n}(M\tenB W,M\tenB V)^\vee \ar[rr]_(.6){(M\tenB-)^\vee} & & 
\hatExt_B^{-n}(W,V)^\vee }$$
where $\tr_{M^\vee}=$ $\tr_{M^\vee}(V,W)$ and
the vertical maps are the Tate duality isomorphisms.
\end{Theorem}

\begin{Remark}
Let $G$ a finite group and $H$ a subgroup of $G$. Tate duality 
for group cohomology is a canonical isomorphism 
$\hat H^{-n}(G;k)^{\vee}\cong$ $\hat H^{n-1}(G;k)$,
for any integer $n$.  Any $k(G\times G)$-module can be viewed
as a $kG$-$kG$-bimodule through the isomorphism $k(G\times G)\cong$
$kG\tenk (kG)^0$ sending $(x,y)\in$ $G\times G$ to $x\ten y^{-1}$. 
Denote by $\Delta G$ the diagonal
subgroup $\Delta G=$ $\{(x,x)\ |\ x\in G\}$ of $G\times G$.
The induction functor $\Ind_{\Delta G}^{G\times G}$ sends the
trivial $k\Delta G$-module to $\Ind_{\Delta G}^{G\times G}(k)\cong$
$kG$, the latter viewed as a $k(G\times G)$-module with $(x,y)\in$
$G\times G$ acting by left multiplication with $x$ and right
multiplication with $y^{-1}$. Combined with the canonical
isomorphism $G\cong$ $\Delta G$ and the interpretation of 
$k(G\times G)$-modules as $kG$-$kG$-bimodules, this functor sends the
trivial $kG$-module to the $kG$-$kG$-bimodule $kG$, and 
induces a split injective graded algebra homomorphism 
$\delta_G : \hat H^*(G;k)\to \hatHH^*(kG)$; similarly for $H$ 
instead of $G$. Following \cite[4.6, 4.7]{Lintransfer}, the 
restriction and transfer maps between $H^*(G;k)$ and $H^*(H;k)$ 
extend to transfer maps between $HH^*(kG)$ and $HH^*(kH)$ induced 
by the $kG$-$kH$-bimodule $kG_H$ and its dual ${_H{kG}}$; a similar 
statement holds for their Tate analogues. The above theorems can
be used to show the well-known fact that Tate duality identifies the 
$k$-dual of the usual transfer map 
$\tr^G_H: \hat H^{-n}(H;k)\to$  $\hat H^{-n}(G;k)$ with the restriction 
map $\res^G_H : \hat H^{n-1}(G;k)\to$ $\hat H^{n-1}(H;k)$. 
This fact is applied, for instance, in Benson's approach \cite{BenNY}
to Greenlees' local cohomology spectral sequence \cite{Greenlees}.
Theorem \ref{transferTatedual} might provide 
one of the technical ingredients towards constructing similar local 
cohomology spectral sequences in Hochschild cohomology of symmetric 
algebras. The background motivation is the question whether the 
Castelnuovo-Mumford regularity of the Hochschild cohomology of 
symmetric algebras is invariant under separable equivalences 
(cf. \cite[3.1, 3.2]{LinHecke}).
\end{Remark}

The proofs of  the above theorems are formal 
verifications, based on explicit descriptions of Tate duality 
for symmetric algebras (reviewed in \S \ref{Tatedualitysection}) 
and of well-known adjunction maps (reviewed in \S \ref{symmetricsection}).
These are used (in \S \ref{Tateadjsection}) to show that Tate duality 
and adjunction are compatible. After a brief review of transfer maps 
in Tate-Hochschild cohomology (in \S \ref{Tatetransfersection}) 
the proofs of  Theorem \ref{transferTatedual} and 
Theorem \ref{transferVWTatedual} are completed in 
\S \ref{proofsection11} and \S \ref{proofsection13}, respectively.
We conclude with some remarks on extending
results of Benson and Carlson \cite{BeCaTate} on products in negative 
group cohomology to the Hochschild cohomology of symmetric algebras 
in \S \ref{negativeTatesection}. The results in this last section have 
independently been obtained in work of Bergh, Jorgensen, and Oppermann
\cite[\S 3]{BJO}.

\medskip
If not stated otherwise, modules are unitary left
modules. For any $k$-algebra $A$ we denote by $\Mod(A)$ the category
of left $A$-modules, and by $\mod(A)$ the subcategory of finitely 
generated $A$-modules. Right $A$-modules are identified with
$A^0$-modules, where $A^0$ is the opposite algebra of $A$. 
Given two $k$-algebras $A$, $B$, we adopt the convention that
for any $A$-$B$-bimodule $M$, the left and right $k$-vector space
structure on $M$ induced by the unit maps $k\to A$ and $k\to B$ of 
$A$ and $B$ coincide. Thus we may consider the $A$-$B$-bimodule $M$ 
as an $A\tenk B^0$-module or as a right $A^0\tenk B$-modules, 
whichever is more convenient. We denote by $\perf(A,B)$ the category
of $A$-$B$-bimodules which are finitely generated projective as left
$A$-modules and as right $B$-modules. Given three $k$-algebras $A$, $B$, $C$,
an $A$-$B$-bimodule $M$, an $A$-$C$-bimodule $N$, and a $C$-$B$-bimodule
$N'$, we consider as usual $\Hom_A(M,N)$ as a $B$-$C$-bimodule via 
$(b\cdot\varphi\cdot c)(m)=$ $\varphi(mb)c$, where $\varphi\in$ 
$\Hom_A(M,N)$, $b\in$ $B$, $c\in$ $C$, and $m\in$ $M$. Similarly, 
we consider $\Hom_{B^0}(M,N')$ as a $C$-$A$-bimodule via
$(c\cdot\psi\cdot a)(m)=$ $c\psi(am)$, where $\psi\in$ 
$\Hom_{B^0}(M,N')$, $a\in$ $A$, $c\in$ $C$, and $m\in$ $M$.

\section{Background material on Tate duality} \label{Tatedualitysection}

Tate duality for symmetric algebras is a special case of 
Auslander-Reiten duality. We need
an explicit description of Tate duality in order to relate
it to the adjunction units and counits arising in the 
definition of transfer maps on Hochschild cohomology.
This is a specialisation to symmetric algebras of arguments 
and results due to Auslander and Reiten in \cite{ARIII}. 
(By taking into account the Nakayama functor, this description 
yields the corresponding duality for selfinjective algebras,
but we will not need this degree of generality in this paper; 
see e.g. \cite{BJ} for more details). 
Let $A$ be a symmetric $k$-algebra; that is, $A\cong$ $A^\vee$ as
$A$-$A$-bimodules. Choose a symmetrising form $s\in$ $A^\vee$; 
that is, $s$ is the image of $1_A$ under a chosen 
bimodule isomorphism $\Phi : A\cong$ $A^\vee$.  Since $A$ is
generated by $1_A$ as a left or right $A$-module and since
$a\cdot 1_A=a=$ $a\cdot 1_A$, it follows that the map $\Phi$
sends $a\in$ $A$ to the linear form $a\cdot s$ defined by 
$(a\cdot s)(a')=$ $s(aa')=$ $s(a'a)=$ $(s\cdot a)(a')$ for 
all $a'\in$ $A$. For any
two $A$-modules $U$, $V$, we denote by $\Hom_A^{pr}(U,V)$ the
space of $A$-homomorphisms from $U$ to $V$ which factor through
a projective $A$-module, and we set $\Hombar_A(U,V)=$
$\Hom_A(U,V)/\Hom_A^{pr}(U,V)$. The stable module category of
$A$ is the $k$-linear category $\modbar(A)$ having the same
objects as $\mod(A)$, with morphism spaces
$\Hombar_A(U,V)$ for any two finitely generated left $A$-modules,
where the composition of morphisms in $\modbar(A)$ is induced
by the composition of $A$-homomorphisms. 
For any finitely generated left $A$-module $U$ choose a projective
$A$-module $P_U$, a surjective $A$-homomrphism $\pi_U : P_U\to$
$U$, an injective $A$-module $I_U$ and an injective $A$-homomorphism
$\iota_U : U\to$ $I_U$. Set $\Omega_A(U)=$ $\ker(\pi_U)$, and
$\Sigma_A =$ $\coker(\iota_U)$. If no confusion arises, we simply
write $\Omega$ and $\Sigma$ instead of $\Omega_A$ and $\Sigma_A$.
The operators $\Sigma$ and $\Omega$ induce inverse self-equivalences,
still denoted $\Omega$ and $\Sigma$, 
on $\modbar(A)$; these functors do not depend on the choice of
the $(P_U,\pi_U)$ and $(I_U,\iota_U)$ in the sense that any other
choice yields functors which are isomorphic to $\Omega$ and
$\Sigma$ through uniquely determined isomorphisms of functors.
The category $\modbar(A)$, together with the self-equivalence $\Sigma$
and triangles induced by short exact sequences in $\mod(A)$ is
triangulated.
Let $U$, $V$ be finitely generated left $A$-modules. For any
integer $n$ set $\hatExt^n_A(U,V)=$ $\Hombar_A(U,\Sigma^n(V))$.
Tate duality for symmetric algebras states that for any integer
$n$ there is an isomorphism
\begin{Equation} \label{Tatedualitysymm}
$$\hatExt^{n-1}_A(V,U) \cong
\hatExt^{-n}_A(U,V)^\vee\ ,$$
\end{Equation}
which is natural in $U$ and $V$. Equivalently, there is a natural
nondegenerate bilinear form
\begin{Equation} \label{Tatenondeg}
$$\langle -,-\rangle :  \hatExt^{n-1}_A(V,U) \times \hatExt^{-n}_A(U,V) \to k\ .$$
\end{Equation}
The isomorphism \ref{Tatedualitysymm} is equivalent to a natural isomorphism
\begin{Equation} \label{TateHomdual}
$$\Hombar_A(V,\Omega(U))\cong \Hombar_A(U,V)^\vee\ .$$
\end{Equation}
Indeed, the isomorphism \ref{TateHomdual} is the special case 
$n=0$ of the isomorphism \ref{Tatedualitysymm}, and conversely, 
\ref{Tatedualitysymm} follows from \ref{TateHomdual} applied with
$\Omega^n(V)$ instead of $V$. The naturality implies in particular
that \ref{Tatedualitysymm} and \ref{TateHomdual} are isomorphisms
of $\Endbar_A(U)$-$\Endbar_A(V)$-bimodules.  
As mentioned before, we will need an explicit description of the 
isomorphism  \ref{TateHomdual} in order to compare it
to transfer maps and their duals. For $\tau\in$ $\Hom_A(U,A)$
and $v\in$ $V$ defined $\lambda_{\tau,v}\in$ $\Hom_A(U,V)$
by setting $\lambda_{\tau,v}(u)=$ $\tau(v)u$ for all $u\in$ $U$.
Note that $\lambda_{\tau,v}$ is the composition of the map
$\tau : U\to$ $A$ followed by the map $A\to$ $V$ sending
$1_A$ to $v$; in particular, $\lambda_{\tau,v}$ factors
through a projective $A$-module. This yields a map
$$\Phi_{U,V} : \Hom_A(U,A)\tenA V \to \Hom_A(U,V)$$
sending $\tau\ten v$ to the map $\lambda_{\tau,v}$.
The maps $\Phi_{U,V}$ are natural in $U$ and $V$. By the above remarks,
the image of $\Phi_{U,V}$ is equal to $\Hom_A^{pr}(U,V)$, and if $U$ 
is finitely generated projective, then $\Phi_{U,V}$ is an isomorphism.
The map sending $\tau\in$ $\Hom_A(U,A)$ to $s\circ\tau$ is a 
natural isomorphism of right $A$-modules $\Hom_A(U,A)\cong$ 
$U^\vee$. Thus, for any finitely generated projective $A$-module 
$P$ we have an isomorphism
$$P^\vee\tenA V \cong \Hom_A(P,V)$$ 
sending $s\circ\alpha\ten v$ to the map $\lambda_{\alpha,v}$ as 
defined above; that is, to the map $x\mapsto\alpha(x)v$, where
$\alpha\in$ $\Hom_A(P,A)$, $v\in$ $V$, and $x\in$ $P$.
Dualising the left term and applying the standard adjunction
and double duality $P^{\vee\vee}\cong$ $P$ 
yields an isomorphism $(P^\vee\tenA V)^\vee\cong$ $\Hom_A(V,P)$.
Together with the previous isomorphism, we obtain an isomorphism
\begin{Equation} \label{VPdual}
$$\Hom_A(V,P) \cong \Hom_A(P,V)^\vee$$
\end{Equation}
sending $\beta\in$ $\Hom_A(V,P)$ to the unique map $\hat\beta\in$ 
$\Hom_A(P,V)^\vee$ satisfying $\hat\beta(\lambda_{\alpha,v})=$ 
$s(\alpha(\beta(v)))$. In the case $A=$ $k$, viewed as
symmetric algebra with $\Id_k$ as symmetrising form, the
isomorphism \ref{VPdual} becomes the canonical isomorphism
$\Hom_k(V,k)=$ $V^\vee\cong$ $\Hom_k(k,V)^\vee$. Let 
$$\xymatrix{ P_1 \ar[r]^{\delta} & P_0 \ar[r]^{\pi} & U \ar[r] & 0}$$
be an exact sequence of $A$-modules, with $P_0=$ $P_U$ and $P_1=$ 
$P_{\Omega(U)}$ projective.
Then $\ker(\pi)=$ $\Im(\delta)=$ $\Omega(U)$. The inclusion 
$\Omega(U)\hookrightarrow$ $P_0$ induces an injective map 
$\Hom_A(V,\Omega(U))\to$ $\Hom_A(V,P_0)$. The map $\delta$ induces 
a map $\Hom_A(V,P_1)\to$ $\Hom_A(V,P_0)$. An $A$-homomorphism
from $U$ to $V$ factors through a projective module if and only if 
it factors through the map $\delta$, viewed as a map from $P_1$ to 
the submodule $\Omega(U)$ of $P_0$. It follows that the image of 
the map $\Hom_A(V,P_1)\to$ $\Hom_A(V,P_0)$ can be identified with 
the subspace $\Hom_A^{pr}(V,\Omega(U))$ of $\Hom_A(V,\Omega(U))$,
where $\Hom_A(V,\Omega(U))$ is viewed as a subspace of $\Hom_A(V,P_0)$. 
Applying the contravariant functor $\Hom_A(-,V)$ to the previous exact 
sequence yields an exact sequence
$$\xymatrix{0\ar[r]&\Hom_A(U,V)\ar[r]&\Hom_A(P_0,V)\ar[r]&\Hom_A(P_1,V)}$$
which remains exact upon applying $k$-duality. Thus we obtain a
commutative diagram
\begin{Equation} \label{Tatedualityproof}
$$\xymatrix{ \Hom_A(V,P_1) \ar[rr]^{\cong}\ar[d] & & 
\Hom_A(P_1,V)^\vee \ar[d] \\
\Hom_A(V,P_0) \ar[rr]^{\cong} & & \Hom_A(P_0,V)^\vee\ar[d] \\
\Hom_A(V,\Omega(U))\ar[u]\ar[rr]_{T} & & 
\Hom_A(U,V)^\vee \ar[d] \\
& & 0 }$$
\end{Equation}
in which the right column is exact, and where the two horizontal
isomorphisms are from \ref{VPdual}. The map $T$ induces the 
desired Tate duality isomorphism. To see this, note first that 
since $\Hombar_A(U,V)$ is a quotient of $\Hom_A(U,V)$, its dual 
can be identified to the subspace of $\Hom_A(U,V)^\vee$ which 
annihilates $\Hom_A^{pr}(U,V)$. The elements in $\Hom_A^{pr}(U,V)$ 
are finite sums of maps of the form $\lambda_{\kappa, v}$, where 
$\kappa\in$ $\Hom_A(U,A)$ and $v\in$ $V$. The image of this space 
in $\Hom_A(P_0,V)$ obtained from precomposing with $\pi$ consists 
of finite sums of maps $\lambda_{\tau,v}$, where $\tau\in$ 
$\Hom_A(P_0,V)$ and $v\in$ $V$ such that $\tau$ factors through 
$\pi$, or equivalently, such that $\tau$ annihilates the submodule 
$\Omega(U)$ of $P_0$. But this is exactly the subspace of all 
$\hat\beta$, where $\beta\in$ $\Hom_A(V,P_0)$ satisfies 
$\Im(\beta)\subseteq$ $\Omega(U)$, and hence $T$ induces a 
surjective map $\Hom_A(V,\Omega(U))\to$ $\Hombar_A(U,V)^\vee$. 
The kernel of this map consists of all homomorphisms in the image 
of the map $\Hom_A(V,P_1)\to$ $\Hom_A(V,P_0)$, which by the above 
is $\Hom_A^{pr}(V,\Omega(U))$. This yields an isomorphism as
stated in \ref{TateHomdual}. A diagram chase shows that this
isomorphism `commutes' with isomorphisms obtained from applying
the equivalence $\Sigma$; that is, the following diagram is
commutative:
\begin{Equation} \label{CY-1}
$$\xymatrix{ \Hombar_A(V,\Omega(U)) \ar[rr] \ar[d] & & 
\Hombar_A(U,V)^\vee \ar[d] \\
\Hombar_A(\Sigma(V), U) \ar[rr] & & 
\Hombar_A(\Sigma(U),\Sigma(V))^\vee } $$
\end{Equation}
where the horizontal isomorphisms are the Tate duality isomorphisms
from \ref{TateHomdual}, where the vertical isomorphisms are
induced by $\Sigma$, and where we have identified $\Sigma(\Omega(U))=$ 
$U=$ $\Omega(\Sigma(U))$ in the lower left corner of this diagram.

\medskip
Using the naturality of Tate duality, we obtain a compatibility of
Tate duality and Yoneda products as follows.
Let $U$, $V$, $W$ be finitely
generated $A$-modules, let $m$, $n$ be integers, and let
$\zeta\in$ $\hatExt^{m+n-1}_A(W,U)$, $\eta\in$ $\hatExt_A^{-m}(V,W)$,
and $\tau\in$ $\hatExt_A^{-n}(U,V)$. Denote by $\zeta\eta=$
$\Sigma^{-m}(\zeta)\circ\eta$  and $\eta\tau=$ $\Sigma^{-n}(\eta)\circ\tau$
the Yoneda products in $\hatExt_A^{n-1}(V,U)$ and
$\hatExt_A^{-m-n}(U,W)$, respectively. Denote by $T(\zeta)$ and $T(\zeta\eta)$
the images of $\zeta$ and of $\zeta\eta$ in $\Ext_A^{-m-n}(U,W)^\vee$ and 
$\Ext_A^{-n}(U,V)^\vee$, respectively, under the appropriate versions of the 
Tate duality isomorphism  \ref{Tatedualitysymm}. We have
\begin{Equation} \label{TateYoneda}
$$T(\zeta\eta)(\tau)= T(\zeta)(\eta\tau)\ ,$$
\end{Equation}
or equivalently,
\begin{Equation} \label{TateYonedanondeg}
$$\langle \zeta\eta, \tau\rangle = \langle \zeta , \eta\tau\rangle\ .$$
\end{Equation}
To see this, consider the diagram
$$\xymatrix{ \hatExt^{n+m-1}_A(W,U) \ar[rr]^{\cong} \ar[d]_{\cong} & &
             \hatExt^{-m-n}_A(U,W)^\vee \ar[d]^{\cong} \\
 \hatExt^{n-1}_A(\Omega^{m}(W),U) \ar[rr]^{\cong} \ar[d]_{(\eta,U)} & & 
 \hatExt_A^{-n}(U,\Omega^{m}(W))^\vee \ar[d]^{(U,\eta)^\vee} \\
 \hatExt_A^{n-1}(V,U) \ar[rr]_{\cong} & & \hatExt_A^{-n}(U,V)^\vee\ }$$
where the horizontal maps are the Tate duality isomorphisms, where the
vertical isomorphisms are induced by $\Omega^{m}$, and where the
two remaining vertical maps are induced by (pre-) composing with $\eta$. 
The upper square is commutative by \ref{CY-1}. The lower square is
commutative by the naturality of Tate duality.
The image of $\zeta$ under the two left vertical maps is
$\zeta\eta$, and the image of $T(\zeta)$ under the right two vertical
maps is the map $\tau\mapsto$ $T(\zeta)(\eta\tau)$. By the
commutativity of this diagram this map is equal to $T(\zeta\eta)$, whence
\ref{TateYoneda} and \ref{TateYonedanondeg}.
Tate duality is dual to its own inverse: combining two Tate duality 
isomorphisms
\begin{Equation} \label{Tatedoubledual}
$$\hatExt^{n-1}_A(V,U)\cong \hatExt_A^{-n}(U,V)^\vee\cong
\hatExt^{n-1}_A(V,U)^{\vee\vee}$$
\end{Equation}
yields the canonical double duality isomorphism, or equivalently, for
$\zeta\in$ $\hatExt_A^{n-1}(V,U)$ and $\eta\in$ $\hatExt^{-n}_A(U,V)$ 
we have
\begin{Equation} \label{Tatedoubledualform}
$$\langle \zeta, \eta\rangle = \langle\eta, \zeta\rangle\ .$$
\end{Equation}

This can be seen by observing that if $P$, $Q$ are two finitely
generated projective $A$-modules, then the composition of the
two consecutive isomorphisms
$\Hom_A(P,Q)\cong$ $\Hom_A(Q,P)^\vee\cong$ $\Hom_A(P,Q)^{\vee\vee}$
obtained from \ref{VPdual} is equal to the canonical double duality
isomorphism.   

\medskip
The opposite algebra $A^0$ of $A$ is again symmetric, with the
same symmetrising form $s$. Thus the algebra $A\tenk A^{0}$ 
is symmetric as well. The Tate-Hochschild 
cohomology of $A$ is defined by $\hatHH^n(A)=$ 
$\Hombar_{A\tenk A^{0}}(A,\Sigma^n(A))$, for any integer $n$, 
where here $\Sigma=$ $\Sigma_{A\tenk A^{0}}$. Tate duality
for Tate-Hochschild cohomology is thus a canonical isomorphism
\begin{Equation} \label{TateHochschildduality}
$$(\hatHH^{-n}(A))^\vee\cong \hatHH^{n-1}(A)\ ,$$ 
\end{Equation}
for any integer $n$.

\section{On adjunction for symmetric algebras}
\label{symmetricsection}

Let $A$, $B$ be symmetric algebras with symmetrising forms
$s\in$ $A^\vee$ and $t\in$ $B^\vee$. The field $k$ is trivially 
a symmetric $k$-algebra, and it is always understood being endowed 
with $\Id_k$ as symmetrising form. Let $M$ be an $A$-$B$-bimodule 
which is finitely generated projective as a left $A$-module 
and as a right $B$-module. It is well-known that the functors 
$M\tenB-$ and $M^\vee\tenA-$ are left and right adjoint to 
each other; see e.g. Brou\'e \cite{Broue1} or \cite[\S 6]{Broue2}. 
We will need the explicit description from \cite{Broue2} of this 
adjunction in order to identify certain isomorphisms as special cases 
of this adjunction. We briefly sketch this, but leave detailed
verifications to the reader. The starting point is the standard 
tensor-Hom-adjunction; this is the adjoint pair
of functors $(M\tenB-, \Hom_A(M,-))$ between $\Mod(A)$ and $\Mod(B)$, 
with the natural isomorphism
$\Hom_A(M\tenB V,U)\cong$ $\Hom_B(V,\Hom_A(M,U))$ sending
$\varphi\in$ $\Hom_A(M\tenB V,U)$ to the map 
$v\mapsto (m\mapsto \varphi(m\ten v))$, where $U$ is an $A$-module,
$V$ a $B$-module, $v\in$ $V$ and $m\in$ $M$. 
The unit of this adjunction is represented by the $B$-$B$-bimodule
homomorphism $B\to$ $\Hom_A(M,M)$ sending $1_B$ to $\Id_M$; the
counit of this adjunction is represented by the $A$-$A$-bimodule
homomorphism $M\tenB \Hom_A(M,A)\to$ $A$ sending $m\ten\alpha$ to
$\alpha(m)$, where $m\in$ $M$ and $\alpha\in$ $\Hom_A(M,A)$.

\medskip
Since $M$ is finitely generated projective as a left $A$-module, 
the canonical map $\Hom_A(M,A)\tenA U\to$ $\Hom_A(M,U)$
sending $\alpha\ten u$ to the map $m\mapsto \alpha(m)u$
is an isomorphism, where $\alpha\in$ $\Hom_A(M,A)$, $u\in$ $U$,
and $m\in$ $M$. Under this isomorphism applied with $U=$ $M$, the
preimage of $\Id_M$ is an expression of the form 
$\sum_{i\in I}\ \alpha_i\ten m_i$, where $I$ is a finite indexing set,
$\alpha_i\in$ $\Hom_A(M,A)$ and $m_i\in$ $M$ such that 
$\sum_{i\in I}\ \alpha_i(m')m_i=$ $m'$ for all $m'\in$ $M$.

\medskip
Since $A$ is symmetric, the map sending $\alpha$ to $s\circ\alpha$ 
is an isomorphism of $B$-$A$-bimodules $\Hom_A(M,A)$ $\cong$ $M^\vee$. 
Similarly, the map sending $\beta\in$ $\Hom_{B^0}(M,B)$ to 
$t\circ\beta$ is an isomorphism of $B$-$A$-bimodules 
$\Hom_{B^0}(M)\cong$ $M^\vee$. Combined with the standard 
adjunction these isomorphisms yield an adjunction
\begin{Equation} \label{MMveeadj}
$$\Hom_A(M\tenB V,U) \cong \Hom_B(V,M^\vee\tenA U)$$
\end{Equation}
sending $\lambda_{\gamma,u}$ to the map $v\mapsto$ $s\circ\gamma_v\ten u$,
where $\gamma\in$ $\Hom_A(M,A)$, $u\in$ $U$, where
$\lambda_{\gamma,u}\in$ $\Hom_A(M\tenB V,U)$ is defined by
$\lambda_{\gamma,u}(m\ten v)=$ $\gamma(m\ten v)u$, and where 
$\gamma_v\in$ $\Hom_A(M,A)$ is defined by
$\gamma_v(m)=$ $\gamma(m\ten v)$, for all $m\in$ $M$, $v\in$ $V$.
The unit and counit of this adjunction are represented by
bimodule homomorphisms 
\begin{Equation} \label{MMveeadjunit}
$$\epsilon_M : B \to M^\vee\tenA M\ ,\ \ 
1_B \mapsto \sum_{i\in I}\ (s\circ \alpha_i) \ten m_i\ ,$$
$$\eta_M : M\tenB M^\vee\to A\ ,\ \ m\ten (s\circ\alpha)\mapsto
\alpha(m)\ ,$$
\end{Equation}
where $I$, $\alpha_i$, $m_i$ are as before.
Similarly, we have an adjunction isomorphism
\begin{Equation} \label{MveeMadj}
$$\Hom_B(M^\vee\tenA U,V) \cong \Hom_A(U,M\tenB V)$$
\end{Equation}
obtained from \ref{MMveeadj} by
exchanging the roles of $A$ and $B$ and using $M^\vee$
instead of $M$ together with the canonical double duality 
$M^{\vee\vee}\cong$ $M$. The adjunction unit and counit
of this adjunction are  represented by bimodule homomorphisms
\begin{Equation} \label{MveeMadjunit}
$$\epsilon_{M^\vee} : A \to M\tenB M^\vee\ ,\ \ 
1_A \mapsto \sum_{j\in J}\ m_j\ten (t\circ \beta_j)\ ,$$
$$\eta_{M^\vee} : M^\vee\tenA M\to B\ ,\ \ (t\circ\beta)\ten m \mapsto
\beta(m)\ ,$$
\end{Equation}
where $J$ is a finite indexing set, $\beta_j\in$ $\Hom_{B^0}(M,B)$,
$m_j\in$ $M$, such that $\sum_{j\in J}\ m_j\beta_j(m')=$ $m'$ for
all $m'\in$ $M$, where $m\in$ $M$ and $\beta\in$ $\Hom_{B^0}(M,B)$.
Note the slight abuse of notation: for the maps $\epsilon_{M^\vee}$
and $\eta_{M^\vee}$ in \ref{MveeMadjunit} to coincide with those
obtained from \ref{MMveeadjunit} applied to $M^\vee$ instead of $M$ 
we need to identify $M$ and $M^{\vee\vee}$. One could avoid this by
replacing the pair of bimodules $(M,M^\vee)$ by a pair of bimodules 
$(M,N)$ which are dual to each other through a fixed choice of a 
nondegenerate bilinear map $M \times N \to$ $k$; this is the point 
of view taken in \cite{Broue2}.

\medskip
The adjunction units and counits of the adjunctions \ref{MMveeadj}
and \ref{MveeMadj} are also the units and counits of the
corresponding adjunctions for right modules. More precisely,
the maps $\epsilon_M$ and $\eta_M$ represent the unit and counit of
the adjoint pair $(-\tenB M^\vee, -\tenA M)$, and the maps
$\epsilon_{M^\vee}$ and $\eta_{M^\vee}$ represent the unit and
counit of the adjoint pair $(-\tenA M, -\tenB M^\vee)$.

\medskip
Duality is compatible with tensor products:
if $N$ is a $B$-$C$-bimodule, where $C$ is another symmetric $k$-algebra,
such that $N$ is finitely generated projective as a left $B$-module
and as a right $C$-module, then we have a natural isomorphism
of $C$-$A$-bimodules
\begin{Equation} \label{dualtensor}
$$N^\vee\tenB M^\vee\cong (M\tenB N)^\vee$$ 
\end{Equation}
sending $(t\circ \beta)\ten\mu$ to the map $m\ten n\mapsto \mu(m\beta(n))$,
where $\mu\in$ $M^\vee$, $\beta\in$ $\Hom_B(N,B)$ (hence $t\circ\beta\in$
$N^\vee$), and where $m\in$ $M$, $n\in$ $N$. This is
obtained as the composition of the natural isomorphisms
$$N^\vee\tenB M^\vee\cong\Hom_B(N,M^\vee)\cong (M\tenB N)^\vee$$
where the second isomorphism is the standard adjunction
with $k$ instead of $A$. Using this isomorphism
(applied to $C=$ $A$ and $N=$ $M^\vee$) we obtain that the adjunction
units and counits from the left and right adjunction of the functors
$M^\vee\tenA-$ and $M\tenB-$ are dual to each other. More
precisely, we have a commutative diagram
\begin{Equation} \label{unitdual}
$$\xymatrix{ A \ar[d] \ar[rr]^(.45){\epsilon_{M^\vee}} & & 
M\tenB M^\vee \ar[d] \\
A^\vee \ar[rr]_(.45){(\eta_M)^\vee} & & (M\tenB M^\vee)^\vee}$$
\end{Equation}
where the left vertical isomorphism is induced by $s$ (sending
$a\in$ $A$ to the linear map $a\cdot s$ defined by $(a\cdot s)(a')=$
$s(aa')$ for all $a\in$ $A$) and where the right vertical
isomorphism combines the isomorphism $(M\tenB M^\vee)^\vee\cong$
$M^{\vee\vee}\tenB M^\vee$ from (\ref{dualtensor}) and the
canonical isomorphism $M^{\vee\vee}\cong$ $M$. The commutativity
is verified by chasing $1_A$ through this diagram. Similarly,
we have a commutative diagram
\begin{Equation} \label{counitdual}
$$\xymatrix{ M^\vee\tenA M \ar[d] \ar[rr]^(.55){\eta_{M^\vee}} & & 
B \ar[d] \\
(M^\vee\tenA M)^\vee \ar[rr]_(.55){(\epsilon_M)^\vee} & & B^\vee}$$
\end{Equation}
where the right vertical isomorphism is induced by $t$ and the
left vertical isomorphism is (\ref{dualtensor})
combined with $M^{\vee\vee}\cong$ $M$ as before.

\medskip
For later reference, 
we mention the special case of the adjunction isomorphism
\ref{MMveeadj} with the algebras $k$, $A$ instead
of $A$, $B$, respectively, the $k$-$A$-bimodule $A$ instead
of $M$, and $k$ and $U$ instead of $U$ and $V$, respectively.
This yields a natural isomorphism
\begin{Equation} \label{UkUAvee}
$$\tau : \Hom_k(U,k)\cong \Hom_A(U,A^\vee)$$
\end{Equation}
sending $\gamma\in$ $\Hom_k(U,k)$ to the map 
$u\mapsto (a\mapsto \gamma(au))$, where $u\in$ $U$ and
$a\in$ $A$. Similarly, the special case of the 
adjunction \ref{MveeMadj} with the algebras $k$, $A$
instead of $A$, $B$, respectively, the $A$-$k$-bimodule $A$
instead of $M$, and the modules $U$, $k$ instead of $V$, $U$, 
respectively, yields a natural isomorphism
\begin{Equation} \label{AveeUkU}
$$\beta : \Hom_A(A^\vee,U)\cong \Hom_k(k,U)$$
\end{Equation}
sending $\varphi\in$ $\Hom_A(A^\vee,U)$ to the unique
linear map sending $1\in$ $k$ to $\varphi(s)$.

\section{Tate duality and adjunction}
\label{Tateadjsection}

As in the preceding section,
let $A$, $B$ be symmetric $k$-algebras, with a fixed choice of
symmetrising forms $s\in$ $A^\vee$ and $t\in$ $B^\vee$. Let $M$
be an $A$-$B$-bimodule which is finitely generated projective
as a left $A$-module and as a right $B$-module. 
Tate duality is induced by the isomorphisms in 
\ref{VPdual}, and so we need to show that these are
compatible with the adjunctions from \S \ref{symmetricsection}.

\begin{Lemma} \label{TateAUadj}
Let $U$ be a finitely generated $A$-module.
We have a commutative diagram of $k$-linear isomorphisms
$$\xymatrix{ 
\Hom_A(U,A) \ar[r]^{\sigma} \ar[d]  
& \Hom_k(U,k) \ar[d] \ar[r]^{\tau} 
& \Hom_A(U,A^\vee)\ar[d]\\
\Hom_A(A,U)^\vee \ar[r]_{\alpha^\vee} 
&  \Hom_k(k,U)^\vee \ar[r]_{\beta^\vee} 
& \Hom_A(A^\vee,U)^\vee } $$
where the map $\sigma$ is induced by composing with the 
symmetrising form $s$, the map $\tau$ is the adjunction
isomorphism,
the map $\alpha^\vee$ is the dual of the canonical
isomorphism $\alpha : \Hom_k(k,U)\cong$ $U\cong$ $\Hom_A(A,U)$,
the map $\beta^\vee$ is the dual of the adjunction isomorphism,
and where the vertical isomorphisms are given by
\ref{VPdual}, with $k$ considered as symmetric algebra
having $\Id_k$ as symmetrising form.
\end{Lemma}

\begin{proof}
Let $\varphi\in$ $\Hom_A(U,A)$. The left vertical isomorphism
sends $\varphi$ to the map 
$\lambda_{\alpha,u}\mapsto s(\alpha(\varphi(u)))$,
where $u\in$ $U$, $\alpha\in$ $\Hom_A(A,A)$ and
$\lambda_{\alpha,u}(a)=$ $\alpha(a) u=$ $a\alpha(1)u$.
Thus $\lambda_{\alpha,u}=$ $\lambda_{\Id,u'}$, where
$\Id$ is the identity map on $A$ and $u'=$ $\alpha(1)u$.
It follows that the left vertical isomorphism sends
$\varphi$ to the unique map sending $\lambda_{\Id,u}$
to $s(\varphi(u))$. 
The upper horizontal isomorphism sends $\varphi$ to
$s\circ\varphi$. Similarly, the middle vertical isomorphism sends
$s\circ\varphi$ to the map sending $\lambda_{\Id_k,u}$ 
to $s(\varphi(u))$. This shows the commutativity of the left
square in the diagram. The commutativity of the right square
can be verified directly using the explicit descriptions of 
$\tau$ and $\beta$ from \ref{UkUAvee} and \ref{AveeUkU}.
Alternatively, it is easy to see that $\tau\circ\sigma$ and 
$\alpha\circ\beta$ are both induced by the isomorphism $A\cong$ 
$A^\vee$ sending $1_A$ to $s$. Thus the outer rectangle (that is, 
with the vertical arrow in the middle removed) is commutative by 
the naturality of the isomorphism \ref{VPdual}. Since
all involved maps are isomorphism, the commutativity of the
right square follows.
\end{proof}

\begin{Lemma} \label{Tateprojadj}
Let $P$ be a finitely generated projective $A$-module, and
let $V$ be a finitely generated $B$-module. Then $M^\vee\tenA P$ 
is  a finitely generated projective $B$-module, and we have
a commutative diagram of $k$-linear isomorphisms
$$\xymatrix{
\Hom_A(M\tenB V,P) \ar[rr] \ar[d]& & \Hom_B(V,M^\vee\tenA P) \ar[d]\\
\Hom_A(P, M\tenB V)^\vee \ar[rr] & & \Hom_B(M^\vee\tenA P, V)^\vee }$$
where the horizontal isomorphisms are given by the adjunction
isomorphisms, and where the vertical isomorphisms are from
\ref{VPdual}.
\end{Lemma}

\begin{proof}
The maps in this diagram are natural in $P$. Thus it suffices
to show the commutativity for $P=$ $A$. More explicitly,
we will show the commutativity of the following
diagram of linear isomorphisms: 

$$\xymatrix{
\Hom_A(M\tenB V, A) \ar[r]\ar[d] & \Hom_k(M\tenB V,k)\ar[r] \ar[d]
& \Hom_B(V,M^\vee) \ar[d] \\
\Hom_A(A,M\tenB V)^\vee\ar[r] & \Hom_k(k,M\tenB V)^\vee\ar[r]
&\Hom_B(M^\vee, V)^\vee }$$
The vertical arrows are isomorphisms from \ref{VPdual}.
The upper two horizontal maps are adjunction isomorphisms, and
their composition is the upper isomorphism of the diagram
in the statement (with $P=$ $A$). Similarly, the lower two 
horizontal maps are dual to adjunction isomorphisms, and their 
composition is the lower horizontal isomorphism of the diagram 
in the statement (with $P=$ $A$).
The commutativity of the left square in this diagram follows 
from that of the left square in Lemma \ref{TateAUadj}. For the 
commutativity it suffices, by naturality, to show this
for $M=$ $B$, which is a special case of the right square in
Lemma \ref{TateAUadj}. 
\end{proof}

\begin{Proposition} \label{Tateadj}
Let $U$ be a finitely generated $A$-module, and
let $V$ be a finitely generated $B$-module. 
We have a commutative diagram of $k$-linear isomorphisms
$$\xymatrix{
\Hombar_A(M\tenB V,\Omega(U)) \ar[rr] \ar[d]& 
& \Hombar_B(V,M^\vee\tenA \Omega(U)) \ar[d]\\
\Hombar_A(U, M\tenB V)^\vee \ar[rr] & & \Hombar_B(M^\vee\tenA U, V)^\vee }$$
where the horizontal isomorphisms are given by the adjunction
isomorphisms, the vertical isomorphisms are the
Tate duality isomorphisms from \ref{TateHomdual}, and
where we identify $\Omega(M^\vee\tenA V)=$ $M^\vee\tenA\Omega(U)$,
with $\Omega$ denoting either $\Omega_A$ or $\Omega_B$.
\end{Proposition}

\begin{proof}
Let 
$$\xymatrix{ P_1 \ar[r]^{\delta} & P_0 \ar[r]^{\pi} & U \ar[r] & 0}$$
be an exact sequence of $A$-modules, with $P_0=$ $P_U$ and $P_1=$ 
$P_{\Omega(U)}$ projective, so that $\ker(\pi)=$ $\Im(\delta)=$ 
$\Omega(U)$. Since $M^\vee$ is finitely generated as a right
$A$-module, the sequence of $B$-modules
$$\xymatrix{ M^\vee\tenA P_1 \ar[r]^{\Id\ten \delta} 
& M^\vee\tenA P_0 \ar[r]^{\Id\ten \pi} & M^\vee\tenA U \ar[r] & 0}$$
is exact. Since $M^\vee$ is also finitely generated as a left
$B$-module, it follows that the $B$-modules $M^\vee\tenA P_1$
and $M^\vee\tenA P_0$ are finitely generated projective.
In particular, we may identify $\Omega(M^\vee\tenA U)=$
$\ker(\Id\ten\pi) =$ $M^\vee\tenA\Omega(U)$.
Combining the commutative diagram \ref{Tatedualityproof}, used
twice (with $A$, $U$, $M\tenB V$ and with $B$, $M^\vee\tenA U$, $V$,
respectively), with the commutative square from \ref{Tateprojadj}, also
used twice (with $P_1$ and $P_0$ instead of $P$), yields the result.
\end{proof}

For the proof of Theorem \ref{transferTatedual} we will need the
following bimodule version of \ref{Tateadj}.

\begin{Corollary} \label{bimoduleTateadj}
Let $C$ be a symmetric $k$-algebra with a fixed choice of
a symmetrising form, $U$ a finitely generated
$A\tenk C^0$-module, and let $V$ be a finitely generated
$B\tenk C^0$-module. We have a commutative diagram of
$k$-linear isomorphisms
$$\xymatrix{
\Hombar_{A\tenk C^0}(M\tenB V,\Omega(U)) \ar[rr] \ar[d]& 
& \Hombar_{B\tenk C^0}(V,M^\vee\tenA \Omega(U)) \ar[d]\\
\Hombar_{A\tenk C^0}(U, M\tenB V)^\vee \ar[rr] & & 
\Hombar_{B\tenk C^0}(M^\vee\tenA U, V)^\vee }$$
where the horizontal isomorphisms are the canonical adjunction
isomorphisms, the vertical isomorphisms are the Tate duality
isomorphisms from \ref{TateHomdual}, and where we identify
$\Omega(M^\vee\tenA U)=$ $M^\vee\tenA \Omega(U)$, 
with $\Omega$ denoting either $\Omega_{B\tenk C^0}$ or 
$\Omega_{A\tenk C^0}$.
\end{Corollary}

\begin{proof}
We will show that this diagram is isomorphic to the 
commutative diagram from \ref{Tateadj}
applied to the algebras $A\tenk C^0$, $B\tenk C^0$
instead of $A$, $B$, respectively, and to the
$A\tenk C^0$-$B\tenk C^0$-bimodule $M\tenk C$,
respectively. 
In this commutative diagram, we identify the
terms through the following isomorphisms.
We consider $M\tenk C$ as an 
$(A\tenk C^0)$-$(B\tenk C^0)$-bimodule
as follows. The left $A\tenk C^0$-module structure on 
$M\tenk C$ is given by left multiplication with $A$ on 
$M$ and by right multiplication with $C$ on $C$.
Similarly, the right $B\tenk C^0$-module structure on
$M\tenk C$ is given by right multiplication with $B$ on
$M$ and by left multiplication with $C$ on $C$. 
We have an isomorphism of $B\tenk C^0$-modules
$(M\tenk C)\ten_{B\tenk C^0} V\cong$ $M\tenB V$
sending $(m\ten c)\ten v$ to $m\ten vc$, where
$m\in$ $M$, $v\in$ $V$, and $c\in$ $C$. Moreover,
since $C$ is symmetric, the choice of a symmetrising
form on $C$ yields an isomorphism of 
$(B\tenk C^0)$-$(A\tenk C^0)$-bimodules 
$(M\tenk C)^\vee\cong$ $M^\vee\tenk C^\vee\cong$ 
$M^\vee\tenk C$. This, in turn, yields an isomorphism of 
$B\tenk C^0$-modules 
$(M\tenk C)^\vee\ten_{A\tenk C^0} \Omega(U)\cong$
$M^\vee\tenA\Omega(U)$.  
With these identifications, the commutative diagram
under consideration takes the form as stated.
\end{proof}

\section{Transfer in Tate-Hochschild cohomology}
\label{Tatetransfersection}

Following \cite{Ligrblock},  a pair of adjoint functors between 
triangulated categories induces transfer maps between the graded 
centers of these categories as well as $\Ext$-groups. We briefly 
review this, specialised to Tate-Hochschild cohomology 
(cf. \cite[\S 7.1]{Ligrblock}). 
Let $A$, $B$ be symmetric $k$-algebras, with a fixed choice of
symmetrising forms $s\in$ $A^\vee$ and $t\in$ $B^\vee$, and let $M$
be an $A$-$B$-bimodule which is finitely generated projective
as a left $A$-module and as a right $B$-module.  Let $n$ be an integer.
We will write $\Sigma$ instead of $\Sigma_{A\tenk A^{0}}$ or
$\Sigma_{B\tenk B^{0}}$. An element $\zeta\in$ $\hatHH^n(B)$
is represented by a $B$-$B$-bimodule homomorphism, abusively denoted
by  the same letter,  $\zeta : B\to$ $\Sigma^n(B)$. We denote
by $\tr_M(\zeta)$ the element in $\hatHH^n(A)$ represented
by the $A$-$A$-bimodule homomorphism
$$\xymatrix{
M\tenB M^\vee = M\tenB B\tenB M^\vee
\ar[rrr]^(.45){\Id_M\ten\zeta\ten\Id_{M^\vee}} & & &
M\tenB \Sigma^n(B)\tenB M^\vee=\Sigma^n(M\tenB M^\vee) }$$
precomposed with the adjunction unit $\epsilon_{M^\vee}: A\to$
$M\tenB M^\vee$ and composed with the `shifted' adjunction counit
$\Sigma^n(\eta_M) : \Sigma^n(M\tenB M^\vee)\to$ $\Sigma^n(A)$.
The identification $M\tenB \Sigma^n(B)\tenB M^\vee=$
$\Sigma^n(M\tenB M^\vee)$ is to be understood as the
canonical isomorphism in $\modbar(A\tenk A^{0})$, using the
fact that the functor $M\tenB - \tenB M^\vee$ sends a
projective resolution of the $B$-$B$-bimodule $B$ to
a projective resolution of the $A$-$A$-bimodule $M\tenB M^\vee$.
Modulo this identification, we thus have
$$\tr_M(\zeta) = \Sigma^n(\eta_M)\circ
(\Id_M\ten\zeta\ten\Id_{M^\vee})\circ \epsilon_{M^\vee}\ .$$
In this way, $\tr_M$ becomes a graded $k$-linear but not necessarily
multiplicative map from $\hatHH^*(B)$ to $\hatHH^*(A)$. We will
need the following alternative description of transfer maps.

\begin{Lemma} \label{transferLemma}
For any integer $n$, the transfer map $\tr_M$ makes the following 
diagram commutative:
{\small
$$
\xymatrix{\Hombar_{B\tenk B^0}(B,\Sigma^n(B)) \ar[d] 
\ar[rr]^{\tr_M}
  &  & \Hombar_{A\tenk A^0}(A,\Sigma^n(A)) \\
\Hombar_{B\tenk B^0}(M^\vee\tenA M,\Sigma^n(B)) \ar[r]_-{\cong}
&  \Hombar_{A\tenk B^0}(M,\Sigma^n(M)) \ar[r]_-{\cong}
&  \Hombar_{A\tenk A^0}(A,\Sigma^n(M)\tenB M^\vee) \ar[u] }
$$
}
where the lower horizontal isomorphisms are adjunction isomorphisms,
the left vertical map us induced by precomposing with the
adjunction counit $M^\vee\tenA M\to$ $B$, and the right vertical
map is induced by composing with the map obtained from applying
$\Sigma^n$ to the adjunction counit $M\tenB M^\vee\to$ $A$.
\end{Lemma}

\begin{proof}
The main theorem on adjoint functors describes adjunction isomorphisms
in terms of adjunction units and counits. Applied to the diagram in
the statement it implies that the composition of the two maps
$$\xymatrix{\Hombar_{B\tenk B^0}(B,\Sigma^n(B)) \ar[r]
&  \Hombar_{B\tenk B^0}(M^\vee\tenA M,\Sigma^n(B)) \ar[r]_{\cong}
&  \Hombar_{A\tenk B^0}(M,\Sigma^n(M))}$$
is equal to the map sending $\zeta\in$ 
$\Hombar_{B\tenk B^0}(B,\Sigma^n(B))$ to $\Id_M\ten\zeta$, where we
identify $M\tenB B=$ $M$ and $\Sigma^n(M)=$ $M\tenB\Sigma^n(B)$,
and where we use abusively the same letters for module homomorphisms 
and their classes in the stable category.
Similarly, the next adjunction isomorphism 
$$\xymatrix{ \Hombar_{A\tenk B^0}(M,\Sigma^n(M)) \ar[rr]_{\cong}
& & \Hombar_{A\tenk A^0}(A,\Sigma^n(M)\tenB M^\vee) }$$
sends $\Id_M\ten\zeta$ to 
$(\Id_M\ten\zeta\ten\Id_{M^\vee})\circ\epsilon_{M^\vee}$, where
$\epsilon_{M^\vee} : A\to$ $M\tenB M^\vee$ is the adjunction counit.
The right vertical map is induced by composition with $\Sigma^n(\eta_M)$,
and hence the image of $(\Id_M\ten\zeta\ten\Id_{M^\vee})\circ\epsilon_{M^\vee}$
is equal to $\Sigma^n(\eta_M)\circ(\Id_M\ten\zeta\ten\Id_{M^\vee})\circ\epsilon_{M^\vee}$.
By the remarks preceding this Lemma, this is equal to $\tr_M(\zeta)$.
\end{proof}

Let $V$, $W$ be finitely generated $B$-modules. An element in
$\hatExt^n_A(M\tenB V,M\tenB W)$ is represented by an $A$-homomorphism
$\eta : M\tenB V\to$ $M\tenB \Sigma^n(W)$, where we identify
$\Sigma^n(M\tenB W)=$ $M\tenB \Sigma^n(W)$ and where we use the same letter
$\Sigma$ for either $\Sigma_A$ or $\Sigma_B$. The transfer map
$\tr_{M^\vee}=$ $\tr_{M^\vee}(V,W)$ sends $\eta$ to the element
$\tr_{M^\vee}(\eta)$ in $\Ext_B^n(V,W)$ represented by the
$B$-homomorphism
$$\xymatrix{V \ar[rr]^(.4){\epsilon_M} & & M^\vee\tenA M\tenB V 
\ar[rr]^(.45){\Id_{M^\vee}\ten\eta} & & M^\vee\tenA M\tenB \Sigma^n(W)
\ar[rr]^(.6){\eta_{M^\vee}} & &  \Sigma^n(W) }$$
The transfer map $\tr_{M^\vee}$ admits the two following descriptions.

\begin{Lemma} \label{transferVWLemma}
For any integer $n$, the transfer map $\tr_{M^\vee}=$ $\tr_{M^\vee}(V,W)$ makes
the following diagram commutative:
$$\xymatrix{ & \Ext_B^n(V,M^\vee\tenA M\tenB W) \ar[rrd]^{(V,\eta_{M^\vee})} & & \\
\Ext_A^n(M\tenB V, M\tenB W) \ar[ru]^{\cong} \ar[rd]_{\cong} \ar[rrr]^{\tr_{M^\vee}} 
& & & \hatExt^n_B(V,W) \\
 & \hatExt_B^n(M^\vee\tenA M\tenB V, W) \ar[rru]_{(\epsilon_M,W)} & & } $$
Here the left two isomorphisms are the adjunction isomorphisms, and the 
maps labelled $(V,\eta_{M^\vee})$ and $(\epsilon_M,W)$ are induced by
composition and precomposition with $\eta_{M^\vee}$ and $\epsilon_M$,
respectively.
\end{Lemma}

\begin{proof}
This follows using the same arguments as in the proof of 
Lemma \ref{transferLemma}.
\end{proof}

\begin{Remark}
Let $n$ be an integer.  The elements of $\hatHH^n(A)$ are morphisms from 
$A$ to $\Sigma^n(A)$ in the stable category $\perfbar(A,A)$ of 
$A$-$A$-bimodules which are finitely generated projective as left and 
right $A$-modules. The category $\perfbar(A,A)$ is a thick subcategory 
of $\modbar(A\tenk A^0)$. 
Following \cite[3.1.(iii)]{Lintransfer} or \cite[5.1]{Ligrblock}, an 
element $\zeta\in$ $HH^n(A)$ is called {\it $M$-stable} if there is
$\eta\in$ $HH^n(B)$ such that $\Id_M\ten\zeta=$ $\eta\ten\Id_M :$ 
$M\to$ $\Sigma^n(M)$ in $\perfbar(A,B)$, where we identify as usual
$\Sigma^n(M)=$ $M\tenB\Sigma^n(B)=$ $\Sigma^n(A)\tenA M$.
Suppose that $M$ and $M^\vee$ induce a stable equivalence of Morita type
between $A$ and $B$. The functor $M\tenB - \tenB M^\vee$ induces
an equivalence of triangulated categories $\perfbar(B,B)\cong$ $\perfbar(A,A)$ 
sending $B$ to the bimodule $M\tenB M^\vee$ which is isomorphic to $A$ in 
$\perfbar(A,A)$. It follows that this functor induces a graded algebra 
isomorphism $\Phi_M : \hatHH^*(B)\cong$ $\hatHH^*(A)$. The adjunction maps
$\epsilon_M$, $\eta_{M}$, $\epsilon_{M^\vee}$, $\eta_{M^\vee}$ are
isomorphisms in the appropriate stable categories of bimodules. Thus every element
in $\hatHH^*(A)$ is $M$-stable.
It follows from \cite[3.6]{Lintransfer} that the isomorphism $\Phi_M$ is 
equal to the analogue for stable categories of the normalised transfer 
map $\Tr_M$ as defined in \cite[3.1.(ii)]{Lintransfer}. 
\end{Remark}

\section{Proof of Theorem \ref{transferTatedual} } \label{proofsection11}

We use the notation from the statement of Theorem \ref{transferTatedual}.
We identify $\Omega\Sigma^n(A)=$ $\Sigma^{n-1}(A)$ 
and $\Omega\Sigma^n(B)=$ $\Sigma^{n-1}(B)$.
The left vertical isomorphism in the diagram in Theorem
\ref{transferTatedual} is the composition
$$\Hombar_{A\tenk A^{0}}(A,\Sigma^{n-1}(A))\cong
\Hombar_{A\tenk A^{0}}(\Sigma^n(A),A))^\vee\cong
\Hombar_{A\tenk A^{0}}(A,\Sigma^{-n}(A)))^\vee\ ,$$
where the first isomorphism is the Tate duality isomorphism
\ref{TateHomdual} applied to $U=$ $\Sigma^n(A)$ and $V=$ $A$, and
where the second isomorphism is induced by the equivalence
$\Sigma^n$ on $\modbar(A\tenk A^{0})$. 
Thus, the commutativity of the diagram in Theorem \ref{transferTatedual}
is equivalent to the commutativity of the diagram
\begin{Equation} \label{transferTatedual2}
$$\xymatrix{ 
\Hombar_{A\tenk A^{0}}(A,\Sigma^{n-1}(A)) \ar[rrr]^{\tr_{M^\vee}} 
\ar[d] & & &  \Hombar_{B\tenk B^{0}}(B,\Sigma^{n-1}(B))  \ar[d] \\
(\Hombar_{A\tenk A^{0}}(\Sigma^n(A),A))^\vee 
\ar[rrr]_{(\Sigma^n \circ \tr_M \circ \Sigma^{-n})^\vee} & & &
(\Hombar_{B\tenk B^{0}}(\Sigma^n(B), B))^\vee  }$$
\end{Equation}
where the vertical maps are appropriate versions of the Tate 
duality isomorphism \ref{TateHomdual}. Lemma \ref{transferLemma}
describes the horizontal maps in this diagram as a composition of
four maps. The commutativity of this diagram will therefore be
established by combining four diagrams as follows. In all four
of those diagrams, the vertical maps are the relevant Tate duality
isomorphisms from \ref{TateHomdual}. Consider the diagram 
\begin{Equation} \label{d1}
$$\xymatrix{ 
\Hombar_{A\tenk A^0}(A,\Sigma^{n-1}(A)) \ar[rr] \ar[d] & & 
\Hombar_{A\tenk A^0}(M\tenB M^\vee, \Sigma^{n-1}(A)) \ar[d] \\
\Hom_{A\tenk A^0}(\Sigma^n(A),A)^\vee \ar[rr] & & 
\Hom_{A\tenk A^0}(\Sigma^n(A),M\tenB M^\vee)^\vee }$$
\end{Equation}
where the horizontal maps are induced by (pre-)composing with the
adjunction counit $M\tenB M^\vee\to$ $A$. The commutativity of
the diagram \ref{d1} follows from the naturality of Tate duality.
Consider next the diagram
\begin{Equation} \label{d2}
$$\xymatrix{ 
\Hombar_{A\tenk A^0}(M\tenB M^\vee,\Sigma^{n-1}(A)) \ar[rr] \ar[d] & & 
\Hombar_{B\tenk A^0}(M^\vee, \Sigma^{n-1}(M^\vee)) \ar[d] \\
\Hom_{A\tenk A^0}(\Sigma^n(A),M\tenB M^\vee)^\vee \ar[rr] & & 
\Hom_{B\tenk A^0}(\Sigma^n(M^\vee),M^\vee)^\vee }$$
\end{Equation}
where the horizontal maps are adjunction isomorphisms, modulo
identifying $M^\vee\tenA\Sigma^{n-1}(A)\cong$ $\Sigma^{n-1}(M^\vee)$
and $M^\vee\tenA\Sigma^n(A)\cong$ $\Sigma^n(M^\vee)$. The
commutativity of \ref{d2} is a special case of the compatibility
\ref{bimoduleTateadj} of Tate duality with adjunction.
Similarly, the analogous version of \ref{bimoduleTateadj} for the
adjoint pair $(-\tenA M,-\tenB M^\vee)$ yields 
the commutativity of the diagram 
\begin{Equation} \label{d3}
$$\xymatrix{ 
\Hombar_{B\tenk A^0}(M^\vee,\Sigma^{n-1}(M^\vee)) \ar[rr] \ar[d] & & 
\Hombar_{B\tenk B^0}(B, \Sigma^{n-1}(M^\vee)\tenA M) \ar[d] \\
\Hom_{B\tenk A^0}(\Sigma^n(M^\vee), M^\vee)^\vee \ar[rr] & & 
\Hom_{B\tenk B^0}(\Sigma^n(M^\vee)\tenA M, B)^\vee }$$
\end{Equation}
where the the horizontal maps are adjunction isomorphisms.
Finally, consider the diagram
\begin{Equation} \label{d4}
$$\xymatrix{ 
\Hombar_{B\tenk B^0}(B, \Sigma^{n-1}(M^\vee)\tenA M)\ar[rr] \ar[d] & & 
\Hombar_{B\tenk B^0}(B,\Sigma^{n-1}(B)) \ar[d]\\
\Hom_{B\tenk B^0}(\Sigma^n(M^\vee)\tenA M, B)^\vee \ar[rr] & &
\Hombar_{B\tenk B^0}(\Sigma^n(B),B) }$$
\end{Equation}
where the horizontal maps are induced by composition with the
adjunction counit $M^\vee\tenA M\to$ $B$, shifted by
$\Sigma^{n-1}$ or $\Sigma^n$ as appropriate. The commutativity
of \ref{d4} follows again from the naturality of Tate duality
\ref{TateHomdual}. Concatenating the four diagrams \ref{d1},
\ref{d2}, \ref{d3}, and \ref{d4} horizontally yields the
commutativity of the diagram \ref{transferTatedual2}, which
completes the proof of Theorem \ref{transferTatedual}.

\section{Proof of Theorem \ref{transferVWTatedual} } \label{proofsection13}

We use the notation from the statement of Theorem \ref{transferVWTatedual}.
It suffices to show the commutativity of the second of the two squares in the
diagram in the statement of \ref{transferVWTatedual}, 
since the first is obtained by duality, thanks to the fact that
applying Tate duality twice yields the canonical double duality.
After replacing $W$ by $\Sigma^n(W)$, it suffices to show the commutativity
of the diagram
\begin{Equation} \label{dgVW1}
$$\xymatrix{ \Hombar_A(M\tenB V, M\tenB \Omega(W)) \ar[rrr]^{\tr_{M^\vee}} \ar[d]
& &  & \Hombar_B(V,\Omega(W)) \ar[d] \\
\Hombar_A(M\tenB W, M\tenB V)^\vee \ar[rrr]_{(M\tenB - )^\vee} 
& & & \Hombar_B(W,V)^\vee}$$
\end{Equation}
where the vertical maps are versions of the Tate duality isomorphism
\ref{TateHomdual}. Lemma \ref{transferVWLemma} describes the map $\tr_{M^\vee}$ 
as a composition of two maps, and hence the commutativity of this diagram
will be established by combining the following two diagrams.
By \ref{Tateadj}, we have a commutative diagram
\begin{Equation} \label{dgVW2}
$$\xymatrix{ \Hombar_A(M\tenB V, M\tenB \Omega(W)) \ar[rr]^{\cong} \ar[d]
&  & \Hombar_B(V,M^\vee\tenA M\tenB \Omega(W)) \ar[d] \\
\Hombar_A(M\tenB W, M\tenB V)^\vee \ar[rr]_{\cong} 
& & \Hombar_B(M^\vee\tenA M \tenB W,V)^\vee}$$
\end{Equation}
where the horizontal isomorphisms are adjunction isomorphisms, and the
vertical isomorphisms are Tate duality isomorphisms.
Using the naturality of Tate duality applied with the couint
$\eta_{M^\vee}$ tensored by either $\Id_W$ or $\Id_{\Omega(W)}$ yields
a commutative diagram
\begin{Equation} \label{dgVW3}
$$\xymatrix{ \Hombar_B(V,M^\vee\tenA M\tenB \Omega(W) \ar[rr] \ar[d] & & 
\Hombar_B(V,\Omega(W)) \ar[d] \\
\Hombar_B(M^\vee\tenA M\tenB W, V)^\vee \ar[rr] & & \Hombar_B(W,V)^\vee }$$
\end{Equation}
Concatenating the two diagrams \ref{dgVW2} and \ref{dgVW3} yields the
diagram \ref{dgVW1}, where we use the description of $\tr_{M^\vee}$ from
Lemma \ref{transferVWLemma}. This proves Theorem \ref{transferVWTatedual}.

\section{Products in negative Tate cohomology} \label{negativeTatesection}

Let $A$ be a symmetric $k$-algebra. The results of this section have 
been obtained independently by Bergh, Jorgensen, and Oppermann 
\cite[\S 3]{BJO}. 
They are generalisations to symmetric algebras of results due to Benson
and Carlson in \cite{BeCaTate}, and the proofs we present here are 
straightforward adaptations of those given in \cite{BeCaTate}.
See also \cite{BeGr} for connections with Steenrod operations.

\begin{Lemma} \label{minusoneproducts}
Let $U$, $V$ be finitely generated $A$-modules, and let $n$ be an integer.
If $\zeta$ is a nonzero element in $\hatExt^{n-1}_A(V,U)$, then there is
a nonzero element $\eta$ in $\hatExt_A^{-n}(U,V)$ such that the
Yoneda product $\zeta\eta$ is nonzero in $\hatExt_A^{-1}(U,U)$.
\end{Lemma}

\begin{proof} 
By Tate duality, if $\zeta$ is nonzero in $\hatExt^{n-1}_A(V,U)$, then
there is $\tau\in$ $\Ext_A^{-n}(U,V)$ such that $\langle \zeta, \tau\rangle\neq$ 
$0$. Denote by $\iota_U$ the image of $\Id_U$ in $\Endbar_A(U)=$
$\hatExt_A^0(U,U)$. Applying the appropriate version of \ref{TateYoneda} 
shows that $\langle \zeta\eta, \iota_U\rangle=$ $\langle \zeta , \eta\rangle\neq$
$0$, hence in particular, $\zeta\eta\neq$ $0$.
\end{proof}

\begin{Lemma} \label{nonzeroproducts}
Let $U$, $V$, $W$ be finitely generated $A$-modules, and let $m$, $n$ be 
integers. Let $\zeta\in$ $\hatExt^{n-1}_A(V,U)$ and $\eta\in$ $\hatExt_A^{m-1}(W,V)$
such that the Yoneda product $\zeta\eta$ is nonzero in $\hatExt_A^{m+n-2}(W,U)$.
Then there is $\tau\in$ $\hatExt^{-m-n+1}_A(U,W)$ such that the
Yoneda product $\eta\tau$ is nonzero in $\hatExt_A^{-n}(U,V)$.
\end{Lemma}

\begin{proof}
By Lemma \ref{minusoneproducts} there is $\tau$ such that $\zeta\eta\tau$ is
nonzero in $\hatExt_A^{-1}(U,U)$. Then necessarily $\eta\tau$ is nonzero, whence
the result.
\end{proof}

For $U$, $V$ two $A$-modules, we denote by $\barExt^*_A(U,V)$ the nonnegative
part of $\hatExt_A^*(U,V)$. That is, for $n>0$ we have
$\barExt_A^n(U,V)=$ $\hatExt^n_A(U,V)=$ $\Ext_A^n(U,V)$, for $n<0$ we
have $\barExt_A^n(U,V)=$ $\Ext_A^n(U,V)=$ $\{0\}$, and
$\barExt_A^0(U,V)=$ $\hatExt_A^0(U,V)=$ $\Hombar_A(U,V)$, while
$\Ext_A^0(U,V)=$ $\Hom_A(U,V)$. 

\begin{Proposition} \label{depthone}
Let $U$ be a finitely generated $A$-module such that $\barExt^*_A(U,U)$ is
graded-commutative. Suppose that there are negative
integers $m$, $n$ such that $\hatExt_A^m(U,U)\cdot\hatExt^n_A(U,U)\neq$ $0$.
Then $\barExt_A^*(U,U)$ has depth at most one.
\end{Proposition}

\begin{proof}
We follow the proof of \cite[Theorem 3.1]{BeCaTate}.
Suppose that $\Ext^*_A(U,U)$ has a regular sequence
of length $2$, consisting of homogeneous elements
$\zeta_1$, $\zeta_2$ of positive degrees $d_1$, $d_2$, respectively.
Let $\zeta\in$ $\hatExt_A^m(U,U)$ and $\eta\in$ $\hatExt_A^n(U,U)$
such that $\zeta\eta\neq$ $0$. By Lemma \ref{nonzeroproducts}, applied
with $m+1$, $n+1$ instead of $m$, $n$, respectively, there is
an element $\tau\in$ $\hatExt_A^{-m-n-1}(U,U)$ such that $\eta\tau\neq$ is
nonzero in $\hatExt^{-m-1}_A(U,U)$. Note that since $m$ is negative,
we have $-m-1\geq$ $0$. Since $\zeta_1$ is not a zero divisor, we 
have $\zeta_1^a\zeta\eta\neq$ $0$, hence $\zeta_1^a\zeta\neq$ $0$
for all nonnegative integers $a$. Choose $a$ maximal such that
$\deg(\zeta_1^a\zeta)<$ $0$. Then $0\leq$ $\deg(\zeta_1^{a+1}\zeta)<d_1$.
In particular, $\zeta^{a+1}\zeta$ is contained in $\barExt_A^*(U,U)$
but not in the ideal $\zeta_1\barExt_A^*(U,U)$. 
Thus the image of $\zeta_1^{a+1}\zeta$ in the quotient
$\barExt_A^*(U,U)/\zeta_1\barExt_A^*(U,U)$ is non zero.
However, for some sufficiently large integer $b$ we have
$\zeta_2^b\zeta_1^{a+1}\zeta=$ $\zeta_1(\zeta_1^a\zeta\zeta_2^b)$, 
hence $\zeta_2$ is not regular on this quotient. This contradiction 
shows that $\barExt_A^*(U,U)$ has no regular sequence of length two.
\end{proof}
 
Since the Hochschild cohomology of an algebra is graded-commutative
we obtain the following consequence:

\begin{Corollary} \label{HHdepthone}
Suppose that there are negative integers $m$, $n$ such that 
$\hatHH^m(A)\cdot\hatHH^n(A)\neq$ $0$.
Then $\barHH^*(A)$ has depth at most one.
\end{Corollary}

\end{document}